\documentclass[12pt,a4paper, reqno]{amsart}
\usepackage{amssymb}
\usepackage{fontenc}
\usepackage{amsfonts}
\usepackage{amsxtra}
\usepackage{amsmath}
\usepackage{amscd}
\usepackage{amsthm}
\usepackage{mathrsfs}
\usepackage{hyperref}
\usepackage{url}
\usepackage{multirow}

\numberwithin{equation}{section}
\newtheorem*{theorem*1}{Theorem A}
\newtheorem*{theorem*2}{Theorem B}
\newtheorem*{proposition*1}{Proposition C}
\newtheorem{lemma}{Lemma}[section]

\newtheorem{remark}[lemma]{Remark}

\newtheorem{theorem}[lemma]{Theorem}
\newtheorem{definition}[lemma]{Definition}

\newtheorem{corollary}[lemma]{Corollary}

\newtheorem*{question*}{Question}
\newtheorem*{assumption*}{Assumption}

\baselineskip 20pt \textwidth 18cm  \theoremstyle{plain}

\newcommand{\Aut}{\operatorname{Aut}}
\newcommand{\End}{\operatorname{End}}
\newcommand{\Hom}{\operatorname{Hom}}

\newcommand{\Irr}{\operatorname{Irr}}

\newcommand{\Ind}{\operatorname{Ind}}
\newcommand{\Coind}{\operatorname{Coind}}
\newcommand{\Rep}{\operatorname{Rep}}

\newcommand{\Id}{\operatorname{Id}}

\newcommand{\Inn}{\operatorname{Inn}}

\newcommand{\C}{\mathbb C}

\newcommand{\GL}{\operatorname{GL}}

\newcommand{\Mn}{\operatorname{M_n}}

\newcommand{\diag}{\operatorname{diag}}

\newcommand{\id}{\operatorname{Id}}

\newcommand{\Ann}{\operatorname{Ann}}

\newcommand{\tr}{\operatorname{Tr}}

\begin{document}
\title{Some conditions from   a  finite  regular semigroup to a finite inverse semigroup}
\date{\today}
\author{Chun-Hui wang }
\address{School of Mathematics and Statistics\\Wuhan University \\Wuhan, 430072,
P.R. CHINA}
\email{cwang2014@whu.edu.cn}
\begin{abstract}
In 1978, Munn proved that a bounded complex representation of an inverse semigroup is  semiunitary and completely reducible. We consider the converse question in the finite case. We  provide some sufficient conditions from   a  finite  regular semigroup to a finite inverse  semigroup.
\end{abstract}
\maketitle
\tableofcontents
\section{Some known results and consequences}
\subsection{A brief introduction}
The purpose of this paper is to prove some conditions sufficiently for  a  finite semigroup to be an inverse semigroup.  We mainly  look for  conditions from complex representation theory of finite semigroups.  As  McAlister's early  survey paper \cite{Mc1} wrote, representations of semigroups have studied for a long time by a lot of pioneering mathematicians like Clifford, Munn, Ponisovsky, Lallement, Petrich, Preston, McAlister, etc.  Here we shall mainly  follow  Steinberg's book\cite{Stein} and  \cite{GMS} to approach this theory. For the recent development,   one can see the references in that book or the related papers in arXiv.  To be precise,  in 1978, Munn proved that a bounded complex representation of an inverse semigroup is  semiunitary and completely reducible. For later use, we consider the converse question for finite semigroups.  We state our  main results  in section \ref{theresults}.  For the proofs,  one can see section \ref{proofs}. By using the structures of finite semigroups, we  finally  reduce the proofs for some  Ree's matrix semigroups;  for such semigroups the results can be proved directly.  We also touch the involution on a semigroup. For such part,  one  can  see  some recent   articles,  for examples \cite{CD},\cite{Dol},\cite{Do2},\cite{EaNo},\cite{Ev},\cite{Sc},etc.
\subsection{Notation and conventions}
We mainly follow the notations of  \cite{Stein}  and  \cite{GMS}.
Let $S$ be a finite semigroup. Let $S^1=S\cup\{1\}$ by adjoining the identity if necessary.      For $S$, let $E(S)$ be  the set of idempotent elements of $S$.  Let  $\mathcal{R}$, $\mathcal{L}$, $\mathcal{J}$   denote the usual Green's relations on $S$. For two elements $a, b\in S$,  $a\mathcal{L} b$ if $S^1 a=S^1 b$,  $a\mathcal{R} b$ if $ aS^1 = bS^1$,   $a\mathcal{J} b$ if  $S^1 aS^1=S^1 bS^1$.  For $a\in S$, let $L_a$(resp. $R_a$, $J_a$) denote the set of generators of $S^1a$ (resp. $aS^1$, $S^1aS^1$) in $S$. Note that  $b\in L_a$  implies   $Sa=Sb$.  For  $e\in E(S)$, let $G_e$ denote the group  of the units of the monoid $eSe$. Let $G_e^0=G_e\cup\{0\}$.   Then $G_e=L_e \cap R_e$. Let us write $L_e=\sqcup_{i=1}^{s_e} x_i\circ_e G_e$, $R_e=\sqcup_{j=1}^{t_e} G_e\circ_e y_j$, $J_e=\sqcup_{i=1}^{s_e}\sqcup_{j=1}^{t_e} x_i\circ_e G_e \circ_e y_j$. Follow the notation of \cite[p.71]{Stein}, let  us define $p_{ji}= y_j \Diamond x_i \in G_e^0$. These elements $p_{ji}$  form a matrix  $P=(p_{ji})$ of type $t_e\times s_e$,   called a sandwich matrix of $J_e$.

 An involution on $S$ is a bijective map $\ast: S \longrightarrow S$, such that (1) $a^{\ast \ast}=a$, (2) $(ab)^{\ast}=b^{\ast}a^{\ast}$, for  $a, b\in S$.  In this case, we will call $S$ an involutive semigroup, denoted by $(S,\ast)$. If for each element $a\in S$, there exists a unique element $a^{\ast}$, such that $aa^{\ast}a=a$, $a^{\ast}aa^{\ast}=a^{\ast}$, we call $S$ an inverse semigroup, and  $a^{\ast}$ is called the inverse of $a$.  In this case, $a\longrightarrow a^{\ast}$,  defines an involution on $S$.

 Let $V$ be a finite dimensional vector space over $\C$.  Let $\pi: S \longrightarrow \End_{\C}(V)$ be a semigroup homomorphism.  By convention, if the zero element of $S$(once  it contains) is mapped  onto  the zero endomorphism of $V$,  we  call $(\pi, V)$   a  (complex) representation of $S$. By Krull-Schmidt theorem, $\pi$ can be decomposed as a direct sum of indecomposable representations.  If $\pi$ has no non-zero dimensional null indecomposable component, we  call it a proper representation. Through this article, a representation  always means a proper representation unless otherwise stated.   When we write the action of $\End_{\C}(V)$ at  $V$ on the left (resp. right) side, we will call $(\pi, V)$ a  left (resp. right) representation of $S$. For  two left representations $(\pi_1, V_1)$, $(\pi_2, V_2)$ of $S$, a $\C$-linear map $f$ from $V_1$ to $V_2$ is called an intertwining operator from $V_1$ to $V_2$, or simply a $S$-morphism if it satisfies $f(\pi_1(s)v)=\pi_2(s)f(v)$,  for all $s\in S$,  $v\in V_1$.   Let $\Rep_l(S)$  denote the category of all complex finite dimensional left  representations of $S$.  A left representation $(\pi, V)$ of $S$ is called irreducible if $V\neq 0$ and $V$ has no non-trivial subrepresentation.  Let $\Irr_l(S)$  be the set of all equivalence classes of  left irreducible representations of $S$. We can give the similar notions for right representations.  For $(\pi, V) \in  \Rep_l(S)$, following \cite[p.277]{Stein}, we can define a canonical duality    $D(V)=\Hom(V,\C)$. Then  $(D(\pi), D(V))\in  \Rep_r(S)$. For simplicity,  in the whole text, we only consider left complex representations and a representation means a left  representation, unless otherwise stated.

 Let $(\pi, V)$ be  a  representation   of $S$  of dimension $n$. Let $\C^n$ denote the vector space of column vectors of dimension $n$.   If we choose a basis $\{e_1, \cdots, e_n\}$, then every vector $v=(e_1, \cdots, e_n) \alpha$, with a column vector $\alpha\in \C^n$. Then $\pi(a)v=(e_1, \cdots, e_n) A\alpha$, $A\in \Mn(\C)$. Moreover,  $\pi: S\longrightarrow \Mn(\C); a\longmapsto A$,  is a semigroup homomorphism, called a  matrix representation, denoted also by $(\pi, \C^n)$.

Let $\C[S]$ denote the  semigroup algebra(or called contracted algebra) of $S$ over the complex field  $\C$. A left representation $(\pi, V)$ is a  left $\C[S]$-module. One  defines $\Ann_S (V)=\{s\in S\mid \pi(s) V=0\}$. Follow \cite{Stein}, \cite{GMS}, if $\Ann_S(V)=I_{J}$, for some  regular $\mathcal{J}$-class $J_e$, we say that  $(\pi, V)$ has an apex $J$ or $e$.  By \cite[Thm.5]{GMS}, every irreducible representation $(\pi, V)$ of $S$ has an apex. Moreover, by the works of Clifford, Munn, Pomizovski, etc, one can recover $V$ from certain irreducible representation of $G_e$ by using the $\Ind_{G_e}$ and $\Coind_{G_e}$ functors. (See \cite[Part II]{Stein},  \cite[Thm 7]{GMS}, for details.)
\begin{theorem}[Schur's Lemma]
 Let $\pi$ be an irreducible matrix  representation of $S$ of dimension $n$. Then:
 \begin{itemize}
 \item[(1)] $\pi: \C[S] \longrightarrow \Mn(\C)$,  is surjective,
 \item[(2)] $\dim \Hom_S(\pi, \pi)=1$.
 \end{itemize}
 \end{theorem}
\begin{proof}
See \cite[Thm.1.2, Thm.2.3]{Mc1} or \cite[Thm.5.33]{CP1}.
\end{proof}
Assume now that  $\C[S]$ is a  semi-simple algebra.  Analogue of complex  representations of $p$-adic groups(cf.\cite{BushH}), for an irreducible representation $\pi=\Ind_{G_e} \sigma, V=\Ind_{G_e} W$, we can define  its contragredient representation $(\check{\pi}=\Ind_{G_e} \check{\sigma}, \check{V}=\Ind_{G_e} \check{W})$. Under the semi-simple assumption on $\C[S]$, for any $(\pi, V)\in  \Rep_l(S)$, we can define its contragredient representation $(\check{\pi}, \check{V})$, up to equivalence.

\subsection{Munn's results and consequences} Let us recall some  definitions from Munn's paper \cite{Mu78}. For a matrix $A\in \Mn(\C)$, let $A^{\ast}$ denote  its conjugate transpose. A matrix $A$ is called preunitary if $AA^{\ast}A=A$. Let $\C^n$ denote the vector space of column vectors of dimension $n$,  endowed  with the canonical Hilbert form. A matrix can be treated as an operator on $\C^n$. As an operator, $A$ is called a partial isometry if $\Vert A\alpha\Vert =\Vert \alpha\Vert $, for any $\alpha \in (\ker A)^{\bot} \subseteq \C^n$.
\begin{lemma}
Let $A$ be a  matrix in $\Mn(\C)$. Then the following  are equivalent:
\begin{itemize}
\item[(1)] $A$ is preunitary;
\item[(2)] As an operator, $A$ is  a partial isometry;
\item[(3)]  $A^{\ast}A$  is an idempotent matrix;
\item[(4)] $AA^{\ast}$ is an idempotent matrix.
\end{itemize}
\end{lemma}
\begin{proof}
By \cite[p.88, Prop.]{Do}, (2)(3)(4) are equivalent.\\
$(1)\Rightarrow (3)$  Since $AA^{\ast}A=A$, $A^{\ast}AA^{\ast}A=A^{\ast}A$ and $AA^{\ast}AA^{\ast}=AA^{\ast}$. \\
$(3)\Rightarrow(1)$ Let $H= (A^{\ast}A)^{\frac{1}{2}}$. As  $A^{\ast}A$ is a positive idempotent matrix, $H= A^{\ast}A$.  Let $A=UH$ be a polar decomposition with   a partial isometry $U$, and $\ker U=\ker H$. So $A^{\ast}AA^{\ast}=H^2HU^{\ast}=HU^{\ast}=A^{\ast}$, and then $AA^{\ast}A=A$.
\end{proof}
 If $A$ is  preunitary, so is $A^{\ast}$.  Let $E= (\ker A)^{\bot}$, $F=(\ker A^{\ast})^{\bot}$. Then $A$ defines a unitary isometry from $E$ to $F$.
\begin{definition}
Let $S$ be a finite semigroup.
\begin{itemize}
\item[(1)] Let  $(\pi, \C^n)$  be a   matrix representation of $S$.
\begin{itemize}
\item[(a)]  If $\pi(a)$ is a preunitary matrix for each element $a$ of $S$,  we call $\pi$ a (matrix) semiunitary representation.
\item[(b)] Let $\ast$ be an involution on $S$.  If $\pi(a^{\ast})=\pi(a)^{\ast}$, for any $a\in S$,  we call $\pi$ a (matrix Hilbert) $\ast$-representation.
\end{itemize}
\item[(2)] For a  representation $(\pi, V)$ of $S$ of dimension $n$,    if there exists a basis $\{e_1, \cdots, e_n\}$,
 under such basis the corresponding matrix representation is a  $\ast$-representation, a semiunitary representation,  a semiunitary $\ast$-representation,   we will call $(\pi, V)$ a $\ast$-representation, a semiunitary representation, a semiunitary $\ast$-representation respectively with respect to the basis $\{e_1, \cdots, e_n\}$.
 \end{itemize}
\end{definition}

\begin{theorem}[Munn]\label{Munn1}
 Every bounded  complex  representation of  an  inverse semigroup is semiunitary with respect to some basis.
\end{theorem}
\begin{proof}
See \cite[Thm.2]{Mu78}.
\end{proof}
\begin{theorem}[Munn]\label{Munn2}
 Every  semiunitary  representation of  an  inverse semigroup is completely reducible.
\end{theorem}
\begin{proof}
See \cite[Thm.1]{Mu78}.
\end{proof}
Following the proof there, we can get the  consequence:
\begin{lemma}\label{comre}
 Every  finite dimensional  $\ast$-representation of  an involutive  semigroup $(S, \ast)$ is completely reducible.
\end{lemma}
\begin{proof}
Let $(\pi, V)$ be a finite dimensional  $\ast$-representation of  $(S, \ast)$ of dimension $n$, $n\geq 1$. For simplicity, assume $V=\C^n$, and $\pi$ is a matrix representation. For any $v\in V$, let $v^{\ast}$ denote its conjugate transpose. For any sub-representation  $0\neq W\subseteq V$, let $W^{\bot}=\{ v\in V \mid v^{\ast} w=0, \textrm{ for all } w\in W\}$. For any $s\in S$, $v\in W^{\bot}$, $w\in W$, $[\pi(s)v]^{\ast} w=v^{\ast} \pi(s)^{\ast}w=v^{\ast} [\pi(s^{\ast})w]=0$, so $\pi(s)v \in W^{\bot}$, $ W^{\bot}$ is  $S$-stable. Moreover, $V=W^{\bot} \oplus W$ as $S$-representations. Hence $(\pi, V)$ is completely reducible.
\end{proof}
Analogue of unitary representations of groups, we can  obtain:
\begin{lemma}\label{uniint}
Let  $(\pi, V)$, $(\pi',V')$ be two equivalent   $\ast$-representations of $(S, \ast)$. Then $(\pi,V)$, $(\pi',V')$ can be  equivalent by a unitary intertwining operator.
\end{lemma}
\begin{proof}
Let $T: V \longrightarrow V'$ be an equivalent $S$-isomorphism. For $s\in S$, $T\circ \pi(s)=\pi'(s) \circ T$. Then $T^{\ast} \circ \pi'(s)^{\ast}=\pi(s)^{\ast}\circ T^{\ast}$, $T^{\ast} \circ \pi'(s^{\ast})=\pi(s^{\ast})\circ T^{\ast}$. Let us write $H=(T^{\ast}T)^{\frac{1}{2}}$, and $T=UH$, for some unitary operator $U$ from $V$ to $V'$. Hence $T^{\ast}T\circ  \pi(s)=T^{\ast}\circ \pi'(s) \circ T=\pi(s)\circ T^{\ast}T$. As $H$ is a limit of some polynomials in $H^2$, $H \circ  \pi(s) =\pi(s) \circ H$. Hence
$T\circ \pi(s)=UH\circ \pi(s)=U\circ \pi(s) \circ H=\pi'(s) \circ UH$. As $H$ is surjective, $U\circ \pi(s) =\pi'(s) \circ U$, for all $s\in S$.
\end{proof}
\begin{lemma}\label{finast}
Let $S$ now be  a finite group, and $\ast$ be an involution on $S$. Then the following are equivalent:
\begin{itemize}
\item[(1)] Every finite dimensional  representation of $(S, \ast)$ is isomorphic to a  $\ast$-representation;
\item[(2)]   $\check{\pi} \simeq D(\pi)\circ \ast$, for any $(\pi, V)\in  \Irr_l(S)$;
\item[(3)]  $\check{\pi} \simeq D(\pi)\circ \ast$, for any $(\pi, V)\in  \Rep_l(S)$.
\end{itemize}
\end{lemma}
\begin{proof}
 By the semi-simplicity of complex representations of a finite group, $(2) \Leftrightarrow (3)$. \\ Note that $1^{\ast}=(11)^{\ast}=1^{\ast}1^{\ast}$, so $1^{\ast}=1$. As a consequence, $1=(aa^{-1})^{\ast}=(a^{-1})^{\ast}a^{\ast}$, which implies $(a^{\ast})^{-1}=(a^{-1})^{\ast}$.  \\
$(1)\Rightarrow(2)$
Assume  that a finite dimensional representation  $\pi$ is isomorphic to a matrix $\ast$-representation $\pi'$. For  representation $\pi$, let us denote by $\chi_{\pi}$ the character of $\pi$. Then for any $s\in S$,
$$\chi_{D(\pi)\circ \ast}(s)= \chi_{D(\pi)}(s^{\ast}) =\chi_{\pi}(s^{\ast}) =\chi_{\pi'}(s^{\ast})=\overline{\chi_{\pi'}(s)}=\overline{\chi_{\pi}(s)}=\chi_{\check{\pi}}(s).$$
So $D(\pi)\circ \ast \simeq \check{\pi}$.

$(2)\Rightarrow(1)$
 Let $\iota$ denote the inverse involution on $S$.  For   any $(\pi, V)\in  \Irr(S)$, $D(\pi)\circ \iota \simeq \check{\pi} \simeq D(\pi)\circ \ast$, so $\pi\circ \ast \simeq \pi\circ \iota$.  Assume now  that  $V=\C^n$, and $\pi$  is an   irreducible  unitary matrix representation of $S$. Then $\pi\circ \iota$, $\pi\circ \ast$  are two equivalently irreducible unitary anti-representations of $S$, so there exists a unitary matrix $A\in \GL_n(\C)$ such that  $\pi(s^{\ast})A=A\pi(s^{-1})$, i.e. $\pi(s^{\ast})=A\pi(s^{-1})A^{-1}$. Moreover, $\pi(s)=\pi(s^{\ast\ast})=A\pi((s^{\ast})^{-1})A^{-1}=A\pi((s^{-1})^{\ast}))A^{-1}=A^2\pi(s)A^{-2}$, for all $s\in S$. Hence $A^2=c\Id$, for some $c\in \C^{\times}$. Since $A$ is a unitary matrix, $c\overline{c}=1$. So $(\sqrt{c}^{-1} A)^2=\id$. Then  $\mid\sqrt{c}\mid=1$, and $\sqrt{c}^{-1} A$ is also a unitary matrix.  Replace $A$ by $ \sqrt{c}^{-1} A$.

 Finally, assume that $A$ is a unitary matrix, and $A^{2}=\id$, $\pi(s^{\ast})=A\pi(s^{-1})A^{-1}$. In this case, let $B\in \GL_n(\C)$ such that $BB^{\ast}=A=A^{\ast}=A^{-1}$. We let $\pi'(s) =B^{-1}\pi(s)B$. Then $\pi'(s^{\ast})=B^{-1}\pi(s^{\ast})B=B^{-1}A\pi(s^{-1})A^{-1}B=B^{\ast}\pi(s^{-1})(B^{\ast})^{-1}=B^{\ast}\pi(s)^{\ast}(B^{-1})^{\ast}=\pi'(s)^{\ast}$. For a non-irreducible representation, it reduces to consider its irreducible components.
\end{proof}
\begin{lemma}
\begin{itemize}
\item[(1)] If $S$ is a finite abelian group  or a finite group of odd order,   the involution satisfying  above conditions is  the inverse map.
\item[(2)] If $S$ is a finite  symmetric group or a finite simple group,   the involution satisfying  above conditions is just the inverse map composited with certain inner automorphism.
\end{itemize}
\end{lemma}
\begin{proof}
  Let $\varphi(s)=(s^{-1})^{\ast}$, for $s\in S$. Since $(s^{-1})^{\ast}=(s^{\ast})^{-1}$,  $\varphi\in \Aut(S)$ and $\varphi^2=1$.  Moreover, for any $\pi\in \Irr(S)$, $\pi^{\varphi} \simeq \pi$.\\
  (1) (a) If $S$ is an abelian group, for any $\pi\in \Irr(S)$, $\pi^{\varphi}(s)=\pi(s)$. Then $\varphi(s)=s$, and $\ast=-1$.\\
  (b) If $S$ is a finite group of odd order, we let $H=\{ s\mid \varphi(s)=s\}$. Then $H$ is a subgroup of $S$.  For $x\in S$, let $[x]=\{ sxs^{-1}  \mid s\in S\}$. For any $\pi\in \Irr(S)$, $\pi^{\varphi}(s)=\pi(s)$, so  each $[x]$ is $\varphi$-stable.   Since all conjugates of the finite group $H$ can not cover the whole finite group $S$, there exists at least one $x_0$ such that  $[x_0]\cap H=\emptyset$.
   Assume that $\varphi $ has order $2$. Then the cardinality $\#[x_0]$  of $[x_0]$ is even, which contradicts to $\# [x_0]\mid |S|$. \\
  (2) (a) If  $S=S_n$, a symmetric group on $n$ letters.  By \cite[Coro.7.7]{Rot},  $\Aut(S_n)=\Inn(S_n)$ for   $n\neq  2$, $n\neq 6$; for $n\neq  2,6$,  $\varphi(s)=asa^{-1}$, and $s^{\ast}=as^{-1}a^{-1}$, for certain $a\in S$.  In case $n=2$, $\Aut(S_2)=\{1\}$, the result also holds. In case $n=6$,  if $\varphi$ is an outer automorphism, by \cite[Coro.7.6]{Rot} $\varphi$ sends  a transposition of $S_6$ to a product of three disjoint  transpositions. But $\varphi$  preserves the conjugate classes of $S_6$,  it is impossible. \\
  (b) If $S$ is a non-abelian simple group, by \cite[Thm.7.14]{Rot}, $\Aut(S)=\Inn(S)$.
\end{proof}
\begin{lemma}
Let $S$ now be  a finite group, and $\ast$ is an involution on $S$. If every finite dimensional  representation of $(S, \ast)$ is a unitary $\ast$-representation with respect to some basis, then the involution is just the inverse map.
\end{lemma}
\begin{proof}
Let $(\pi, \C^n)$ be a faithful  matrix unitary $\ast$-representation of $(S, \ast)$. Then for any $s\in S$, $\pi(s^{\ast})=\pi(s)^{\ast}$, $\id=\pi(s)\pi(s)^{\ast}=\pi(ss^{\ast})$. Since $\pi$ is a faithful representation, $ss^{\ast}=1$ and $s^{\ast}=s^{-1}$.
\end{proof}
\subsection{Ree's matrix semigroup over a group with zero}
Let $I$ and  $J$ be the finite sets $\{1, \cdots, m\}$ and $\{1, \cdots, n\}$. Let $G$ be a finite group, and $G^{0}=G \cup \{0\}$.   Let $P$ be a matrix of type $n\times m$  with each entry $p_{ij}\in G^{0}$. For any $a\in G$,  let $(a)_{ij}$ denote the  $I\times J$-Ree's matrix over $ G^{0}$ with  $a$ in the $(i,j)$-entry and $0$ elsewhere. Let $\mathcal{M}^0(I, J, G, P)$ be a finite set consisting  of all elements   $ (a)_{ij}$ and the zero matrix of type  $m\times n$. One defines a binary operator $\circ$ on $ \mathcal{M}^0(I, J, G, P)$ by $(a)_{ij} \circ  (b)_{kl}=(a p_{jk} b)_{il}$. Then $(\mathcal{M}^0(I, J, G, P), \circ )$ is a semigroup, called a Ree's matrix semigroup over $G^{0}$; $P$ is called a sandwich matrix.  It is well known that every completely $0$-simple finite semigroup is isomorphic to a regular Ree's matrix semigroup.  If $I=J$, and $P$ is a non-singular matrix over $\C[G]$, $\C[\mathcal{M}^0(I, J, G, P)]$ is a semi-simple algebra.
\begin{lemma}\label{ismir}
\begin{itemize}
\item[(1)] Two  regular Ree's matrix semigroups  $ \mathcal{M}^0(I, J, G, P)$ and  $ \mathcal{M}^0(I, J, G', P')$ are isomorphic  iff there exist bijective maps $\alpha: I\longrightarrow I$, $ \beta: J \longrightarrow J$,  a group isomorphism $\varphi: G \longrightarrow G'$, and $\varphi(0)=0$, elements $u_i$, $v_j\in G'$, $i\in I$, $j\in J$    such that for the entries $p'_{\beta(j), \alpha(i)}$  in $P'$  and $p_{ji} $ in  $P$, we have $p'_{\beta(j), \alpha(i)}= v_j^{-1}\varphi(p_{ji})u_i^{-1}$.
\item[(2)] $ \mathcal{M}^0(I, J, G, P)$  is an inverse semigroup  iff $I=J$, $P$ can be chosen to be   an identity matrix up to isomorphism.
\end{itemize}
\end{lemma}
\begin{proof}
For (1), see \cite[Thm.2,8]{Ho}. For (2), see \cite[Thm.8.1]{Cl42}.
\end{proof}
In the above (1),  the isomorphism from $ \mathcal{M}^0(I, J, G, P)$ to   $ \mathcal{M}^0(I, J, G', P')$ can be given by $(a)_{ij} \longrightarrow (u_i\varphi(a)v_j)_{\alpha(i), \beta(j)} $.

 Let $S= \mathcal{M}^0(I, J, G, P)$ be a regular semigroup.  Let $(\sigma, \C^k)$ be a matrix representation of $G$. We can define two  matrix representations $\pi^l$, $\pi^r$ of $S$ as follows: for $s=(a)_{ij}\in S$, $\pi^l(s)=\sigma((a)_{ij}P)$, and $\pi^r(s)=\sigma(P(a)_{ij})$;   here we treat $P$, $(a)_{ij}$  as matrices over $G^0$, and $\sigma$ acts on a matrix by acting on its entries. By  \cite[Def.1.1]{LaPe}, we call such $\pi^l, \pi^r$  standard representations of $S$.   By \cite[Prop.1.2, Thm.1.4]{LaPe}, (1) $\pi^r$ is a proper representation, (2) if $\sigma$ is irreducible, $\pi^r$ contains only one  non-null irreducible constituent, (3) every irreducible representation of $S$ is equivalent to one such constituent. By duality, the similar result also holds for $\pi^l$. For any $s=(a)_{ij} \in S$,
 \begin{equation}
 P(a)_{ij}= \bordermatrix{
  &  &        &  & j       &   &       &  \cr
 &0 & \cdots & 0& p_{1i}a &0 & \cdots & 0 \cr
 & 0 & \cdots & 0& p_{2i}a &0 & \cdots & 0 \cr
 &\vdots &    & \vdots  & \vdots &\vdots&  & \vdots \cr
 & 0 & \cdots & 0& p_{ni}a &0 & \cdots & 0 \cr
}_{n\times n}
 \end{equation}
 \begin{equation}
 (a)_{ij}P=   \bordermatrix{
&       &        &            &          \cr
&0      & 0      &   \cdots  & 0    \cr
&\vdots & \vdots &           & \vdots   \cr
&0      & 0      &   \cdots  & 0    \cr
i&ap_{j1}& ap_{j2} & \cdots   & ap_{jm}  \cr
&0      & 0      &   \cdots  & 0    \cr
&\vdots & \vdots &           & \vdots   \cr
&0      & 0      &   \cdots  & 0}_{m\times m}
  \end{equation}
  \begin{equation}
\pi^r(s)=\sigma(P(a)_{ij})= \bordermatrix{
 &  &        &  & j       &   &       &  \cr
&0 & \cdots & 0&\sigma( p_{1i}a) &0 & \cdots & 0 \cr
 & 0 & \cdots & 0&\sigma( p_{2i}a) &0 & \cdots & 0 \cr
 &\vdots &    & \vdots  & \vdots &\vdots&  & \vdots \cr
 & 0 & \cdots & 0& \sigma(p_{ni}a) &0 & \cdots & 0
}_{n\times n}
\end{equation}
 \begin{equation}
\pi^l(s)=\sigma( (a)_{ij}P)= \bordermatrix{
&       &        &            &          \cr
&0      & 0      &   \cdots  & 0    \cr
&\vdots & \vdots &           & \vdots   \cr
&0      & 0      &   \cdots  & 0    \cr
i&\sigma(ap_{j1})& \sigma(ap_{j2}) & \cdots   &\sigma( ap_{jm} ) \cr
&0      & 0      &   \cdots  & 0    \cr
&\vdots & \vdots &           & \vdots   \cr
&0      & 0      &   \cdots  & 0}_{m\times m}
\end{equation}

 \begin{equation}
 \pi^r(s)^{\ast}=[\sigma(P(a)_{ij})]^{\ast}= \bordermatrix{
 &                              &                         &               &                        \cr
& 0                             & 0                       &   \cdots      & 0                      \cr
&\vdots                         & \vdots                  &               & \vdots                 \cr
&     0                         & 0                       &   \cdots      & 0                      \cr
j&\sigma(p_{1i}a)^{\ast}         & \sigma(p_{2i}a)^{\ast}  & \cdots       & \sigma(p_{ni}a)^{\ast} \cr
&0                              & 0                       &   \cdots      & 0                     \cr
&\vdots                         & \vdots                  &               & \vdots                 \cr
&0                              & 0                       &   \cdots      & 0}_{n\times n}
 \end{equation}
 \begin{equation}
\pi^l(s)^{\ast}=[\sigma((a)_{ij}P)]^{\ast}=
 \bordermatrix{
  &  &        &  & i       &   &       &  \cr
 &0 & \cdots & 0&\sigma(ap_{j1})^{\ast}&0 & \cdots & 0 \cr
 & 0 & \cdots & 0&\sigma(ap_{j2})^{\ast} &0 & \cdots & 0 \cr
 &\vdots &    & \vdots  & \vdots &\vdots&  & \vdots \cr
 & 0 & \cdots & 0& \sigma(ap_{jm})^{\ast}&0 & \cdots & 0
}_{m\times m}.
\end{equation}

 Following  the proof of the above lemma \ref{ismir} (1) in  \cite[p.66]{Ho}, up to isomorphism,  we  assume $p_{1j}=1$ or $0$ and $p_{11}=1$. Then $\iota: a\longrightarrow (a)_{11}$, defines a group monomorphism  from $G$ to $S$. Moreover,  $(a)_{ij}=(1)_{i1} \circ (a)_{11} \circ (1)_{1j}$.
 \begin{lemma}\label{preunitary}
 Keep the notations.  If both $\pi^l$, $\pi^r$ are  semiunitary representations, then:
  \begin{itemize}
 \item[(1)] $\sigma$ is a unitary representation of $G$,
 \item[(2)] $S$ is an inverse semigroup.
 \end{itemize}
 \end{lemma}
 \begin{proof}
 1) By $\pi^l\circ \iota(a)=\pi^l\circ \iota(a) [\pi^l\circ \iota(a)]^{\ast} \pi^l\circ \iota(a)$,
  \begin{equation}
 \sigma((a)_{11}P)= \begin{bmatrix}
 \sigma(ap_{11})& \sigma(ap_{12}) & \cdots   &\sigma( ap_{1m})  \\
0      & 0      &   \cdots  & 0   \\
\vdots & \vdots &           & \vdots   \\
0      & 0      &   \cdots  & 0
\end{bmatrix}
  \end{equation}
  $$=\begin{bmatrix}
 \sigma(ap_{11})& \sigma(ap_{12}) & \cdots   &\sigma( ap_{1m})  \\
0      & 0      &   \cdots  & 0   \\
\vdots & \vdots &           & \vdots   \\
0      & 0      &   \cdots  & 0
\end{bmatrix}
 \begin{bmatrix}
 \sigma(ap_{11})^{\ast}    &0   & \cdots & 0 \\
  \sigma(ap_{12})^{\ast}   &0   & \cdots & 0\\
        \vdots             &\vdots   &   & \vdots   \\
 \sigma(ap_{1m})^{\ast}    &0   & \cdots & 0
\end{bmatrix}
\begin{bmatrix}
 \sigma(ap_{11})& \sigma(ap_{12}) & \cdots   &\sigma( ap_{1m})  \\
0      & 0      &   \cdots  & 0   \\
\vdots & \vdots &           & \vdots   \\
0      & 0      &   \cdots  & 0
\end{bmatrix}$$
$$=\begin{bmatrix}
 (\sum_{l=1}^m \sigma(ap_{1l})\sigma(ap_{1l})^{\ast})  \sigma(ap_{11})& (\sum_{l=1}^m \sigma(ap_{1l})\sigma(ap_{1l})^{\ast})\sigma(ap_{12}) & \cdots   &(\sum_{l=1}^m \sigma(ap_{1l})\sigma(ap_{1l})^{\ast})\sigma( ap_{1m})  \\
0      & 0      &   \cdots  & 0   \\
\vdots & \vdots &           & \vdots   \\
0      & 0      &   \cdots  & 0
\end{bmatrix}.$$
Since $S$ is a regular semigroup,  there exists $l$ such that $\sigma(ap_{1l})$  is non-singular. Hence $ \sum_{l=1}^m \sigma(ap_{1l})\sigma(ap_{1l})^{\ast} =E$, for any $a\in G$.  Assume $t=\# \{ l\mid p_{1l}=1\}$. Then $\sigma(a) \sigma(a)^{\ast}=\frac{1}{t}E$. In particular, take $a=1_G$, $t=1$.  Hence $\sigma(a)\sigma(a)^{\ast}=E$, for any $a\in G$. \\
 2) For  any $s=(a)_{ij} \in S$, $\pi^l(s)\pi^l(s)^{\ast}\pi^l(s)=\pi^l(s)$ implies $\sum_{l=1}^m \sigma(p_{jl}a) \sigma(p_{jl}a)^{\ast}=E$. So for each fixed $j$, there exists only one nonzero $p_{jl}$, for  $l=1, \cdots, m$.   Similarly, $\pi^r(s)\pi^r(s)^{\ast}\pi^r(s)=\pi^r(s)$ implies $\sum_{k=1}^n \sigma(p_{ki}a)^{\ast} \sigma(p_{ki}a)=E$. So for each fixed $i$, there exists only one nonzero $p_{ki}$, for  $k=1, \cdots, n$.
 Hence  $P$ is a non-singular sandwich matrix.  Up to isomorphism, $P$ can be chosen to be the identity matrix. By Lemma \ref{ismir}(2), $S$ is an inverse semigroup.
 \end{proof}
Let   $\mathcal{A}=\{ (\sigma_1, V_1), \cdots,  (\sigma_k, V_k)\}$  be a set of all pairwise inequivalent irreducible  matrix  representations of $G$.
 \begin{definition}\label{lr}
 If the above $\sigma=\oplus_{i=1}^k \sigma_i$,  we will call the corresponding $(\pi^l, \pi^r)$,  a pair of   matrix Sch\"utzenberger representations of $S$.
 \end{definition}
Note that by \cite[Chap.5.5]{Stein}, if $\C[S]$ is a semi-simple algebra, $\pi^l \simeq \pi^r$.
\begin{lemma}\label{astrepr}
   \begin{itemize}
   \item[(1)]  If both  $\pi^l$, $\pi^r$ are semiunitary  representations, then $S$ is an inverse semigroup.
   \item[(2)]  If both  $\pi^l$, $\pi^r$ are semiunitary $\ast$-representations, then $S$ is an inverse semigroup with the inverse map given by $\ast$.
   \end{itemize}
   \end{lemma}
\begin{proof}
1) It  is an example of Lemma \ref{preunitary}. \\
2)  For any $s\in S$, $\pi^l(s^{\ast})=\pi^l(s)^{\ast}$, $\pi^l(s)=\pi^l(s)\pi^l(s)^{\ast}\pi^l(s)=\pi^l(ss^{\ast}s)$.  Note that $\pi^l$ is a faithful representation. So $ss^{\ast}s=s$, $s^{\ast}ss^{\ast}=s^{\ast}$. Hence $s^{\ast}$ is the inverse of $s$.
\end{proof}
Following the proof of the above lemma \ref{ismir} (1) in  \cite[p.66]{Ho}, we can  derive the next result:
\begin{lemma}\label{SS}
Let $S=\mathcal{M}^0(I, J, G, P)$ be a regular semigroup. A bijective map  $\ast$ on $S$ defines  an involution
 iff $I=J$, and  there exist
 \begin{itemize}
 \item[(1)]  an involution   $\ast$  on $G$ as well as  $G^0$,
 \item[(2)]  an involution $\varphi$ on  the set $I$,
 \item[(3)] elements $u_i\in G$, $i\in I$,
 \item[(4)] a central element $z\in Z(G)$,  $z^{\ast}=z^{-1}$,
  \end{itemize}
  such that
  \begin{itemize}
   \item[(a)]  $(a)_{ij}^{\ast}=(zu_{\varphi(j)}^{\ast}a^{\ast}u_i^{-1})_{\varphi(j), \varphi(i)}$, for any $a\in G$,
   \item[(b)]  $p_{ji}^{\ast}= zu_i^{-1} p_{\varphi(i), \varphi(j)}u_{\varphi(j)}^{\ast}$.
   \end{itemize}
\end{lemma}
\begin{proof}
$(\Rightarrow)$ If $S=S^1$, then $1^{\ast} a^{\ast}=(a1)^{\ast}=a^{\ast}= (1a)^{\ast}=a^{\ast} 1^{\ast}$, for any $a\in S$. So $1^{\ast}=1$. If $S\neq S^1$, we can extend the involution to define over $S^{1}$ by adding $1^{\ast}=1$.
 Let $S_{ji}=\{ (a)_{ji} \mid a\in G\}$,  $L_{i}=\{ (b)_{pi}\mid p\in I, b\in G\}$, $R_{j}=\{ (b)_{jq}\mid q\in J, b\in G\}$.  Then for $a=(a)_{ji}\in S$, $L_a=L_i$, $R_a=R_j$.   $a\mathcal{L} b$ iff $S^1a=S^1 b$ iff $a^{\ast} S^1=b^{\ast}S^1$ iff  $a^{\ast}\mathcal{R} b^{\ast}$. So $\ast$ maps $L_{a}$ onto $R_{a^{\ast}}$, and $R_{a}$ onto $L_{a^{\ast}}$. Assume $a^{\ast}=(b)_{kl}\in S$. Then $ L_{a^{\ast}}=L_l$, $R_{a^{\ast}}=R_k$, $\ast$ maps $L_{i}$ onto $R_{k}$. Therefore $I=J$. Note that $k$ is determined by $i$, independent of $j$. Let us   write $k=\varphi(i)$. Then $\varphi$ is a bijective map on $I$, and $\varphi^2=id$. Then $\ast$ sends $L_i$ onto $R_{\varphi(i)}$, $R_j$ onto $L_{\varphi(j)}$. Since $L_i\cap R_j= S_{ji}$, $\ast$ sends $S_{ji}$ onto $S_{\varphi(i), \varphi(j)}$.

 For $i_0\in I$,  assume $p_{j_0,i_0}\in G$.  Then  $e_{i_0}= (p_{j_0,i_0}^{-1})_{ i_0,j_0}\in E(S)$, $e_{i_0}^{\ast}\in E(S) \cap S_{\varphi(j_0), \varphi(i_0)}$, so $e_{i_0}^{\ast}=(p_{\varphi(i_0), \varphi(j_0)}^{-1})_{\varphi(j_0),\varphi(i_0)}$. Let us write $S_{i_0,j_0}=\{ (ap_{j_0,i_0}^{-1})_{i_0,j_0}\mid a\in G\}$,
$S_{\varphi(j_0),\varphi(i_0)}=\{ (ap_{\varphi(i_0),\varphi(j_0)}^{-1})_{\varphi(j_0),\varphi(i_0)}\mid a\in G\}$, and
$(ap_{j_0,i_0}^{-1})_{i_0,j_0}^{\ast}=(a^{\ast} p_{\varphi(i_0),\varphi(j_0)}^{-1})_{\varphi(j_0),\varphi(i_0)}$.
Applying  the involution $\ast$ on the equality $(a_1p_{j_0,i_0}^{-1})_{i_0,j_0} \circ (a_2p_{j_0,i_0}^{-1})_{i_0,j_0} =(a_1a_2p_{j_0,i_0}^{-1})_{i_0,j_0}$, we obtain:
\[((a_1a_2)^{\ast}p_{\varphi(i_0),\varphi(j_0)}^{-1})_{\varphi(j_0),\varphi(i_0)}=(a_2^{\ast}p_{\varphi(i_0),\varphi(j_0)}^{-1})_{\varphi(j_0),\varphi(i_0)} \circ  (a_1^{\ast}p_{\varphi(i_0),\varphi(j_0)}^{-1})_{\varphi(j_0),\varphi(i_0)}\]
\[=(a_2^{\ast}a_1^{\ast} p_{\varphi(i_0),\varphi(j_0)}^{-1})_{\varphi(j_0),\varphi(i_0)}. \]
  Hence $(a_1a_2)^{\ast}=a_2^{\ast}a_1^{\ast}$.   For any other $(a)_{ij} \in S_{ij}$, $(a)_{ij}= (p_{j_0,i_0}^{-1})_{i, j_0}\circ  ( ap_{j_0,i_0}^{-1})_{i_0,j_0}\circ (1 )_{i_0,j}$. Let us write $ (p_{j_0,i_0}^{-1})_{i, j_0}^{\ast}=(u_i^{-1})_{\varphi(j_0),\varphi(i)}$, $(1 )_{i_0,j}^{\ast}=(v_jp_{\varphi(i_0), \varphi(j_0)}^{-1})_{\varphi(j),\varphi(i_0)}$, for some $u_i, v_j\in G$.
  Hence
  $$(a)_{ij}^{\ast}= (1 )_{i_0,j}^{\ast} \circ ( ap_{j_0,i_0}^{-1})_{i_0,j_0}^{\ast} \circ(p_{j_0,i_0}^{-1})_{i, j_0}^{\ast}$$
  $$=(v_jp_{\varphi(i_0), \varphi(j_0)}^{-1})_{\varphi(j),\varphi(i_0)} \circ (a^{\ast} p_{\varphi(i_0),\varphi(j_0)}^{-1})_{\varphi(j_0),\varphi(i_0)}\circ (u_i^{-1})_{\varphi(j_0),\varphi(i)}$$
  $$=(v_jp_{\varphi(i_0), \varphi(j_0)}^{-1} p_{\varphi(i_0), \varphi(j_0)} a^{\ast} p_{\varphi(i_0),\varphi(j_0)}^{-1}p_{\varphi(i_0), \varphi(j_0)}u_i^{-1})_{\varphi(j),\varphi(i)}=(v_j a^{\ast} u_i^{-1})_{\varphi(j),\varphi(i)}.$$

  Note that $\ast$ is an involution on $S$.
  \begin{itemize}
  \item[(i)] $(a)_{ij}=(a)_{ij}^{\ast\ast}=[  (v_j a^{\ast} u^{-1}_i)_{\varphi(j),\varphi(i)}]^{\ast}= (v_{\varphi(i)}[ (u_i^{-1})^{\ast}av_j^{\ast} ]u^{-1}_{\varphi(j)})_{ij}$. Hence $v_{\varphi(i)}(u^{-1}_i)^{\ast}av_j^{\ast} u^{-1}_{\varphi(j)}=a$, for all $i, j\in I$. Fix $i$,  choose different $j$, and  vice-versa.  Then   $zu_{i}^{\ast}=v_{\varphi(i)}$, all $i$,  for some $z\in Z(G)$,  $z^{\ast}=z^{-1}$.
  \item[(ii)] $(a)_{ij} \circ (b)_{kl}=(ap_{jk}b)_{il}$, so $(v_l b^{\ast} p^{\ast}_{jk}a^{\ast} u^{-1}_i)_{\varphi(l), \varphi(i)}=(ap_{jk}b)^{\ast}_{il}= (b)_{kl}^{\ast}\circ (a)_{ij}^{\ast}=(v_lb^{\ast}u_k^{-1})_{\varphi(l), \varphi(k)}\circ(v_ja^{\ast}u^{-1}_i)_{\varphi(j), \varphi(i)}=( v_lb^{\ast}u^{-1}_kp_{\varphi(k), \varphi(j)}v_ja^{\ast}u^{-1}_i)_{\varphi(l), \varphi(i)}$. Hence $b^{\ast} p^{\ast}_{jk}a^{\ast}=b^{\ast}u^{-1}_kp_{\varphi(k), \varphi(j)}v_ja^{\ast}$, which implies that $p^{\ast}_{jk}=u^{-1}_kp_{\varphi(k), \varphi(j)}v_j=zu^{-1}_kp_{\varphi(k), \varphi(j)}(u_{\varphi(j)})^{\ast}$.
   \end{itemize}
$(\Leftarrow)$ One can recover the equalities in the above (i), (ii) from the given conditions. So $\ast$ defines an involution on $S$.
\end{proof}

Let us change $S$ by certain isomorphism as described in Lemma \ref{ismir}. Let $p'_{ji}= p_{\beta(j),\alpha(i)}u_{\alpha(i)}^{\ast}$, $(a)_{ij}'=(u_{\alpha(i)}^{\ast} a)_{ \alpha(i), \beta(j)}$, for some bijective maps $\alpha, \beta: I \longrightarrow I$ such that $\varphi=\beta\circ \alpha^{-1}=\varphi^{-1}$, for  $u_i$ in above lemma \ref{SS}. Then
$$(p'_{ji})^{\ast}=u_{\alpha(i)}p^{\ast}_{ \beta(j), \alpha(i)}=(u_{\alpha(i)})zu_{\alpha(i)}^{-1} p_{\varphi\circ\alpha(i), \varphi\circ\beta(j)}u_{ \varphi\circ\beta(j)}^{\ast}$$
$$=zp_{\varphi\circ\alpha(i), \varphi\circ\beta(j)}u_{ \varphi\circ\beta(j)}^{\ast}=zp'_{\beta^{-1}\circ \varphi\circ\alpha(i), \alpha^{-1}\circ  \varphi\circ\beta(j)}=zp'_{ij}, $$
$$[(a)'_{ij}]^{\ast}=[(u_{\alpha(i)}^{\ast} a)_{\alpha(i), \beta(j)}]^{\ast}=(zu_{\varphi\circ \beta(j)}^{\ast}a^{\ast}u_{\alpha(i)}u_{\alpha(i)}^{-1})_{\varphi\circ\beta(j), \varphi\circ\alpha(i)}$$
$$=(zu_{\alpha(j)}^{\ast}a^{\ast})_{\alpha(j), \beta(i)}=z(a^{\ast})'_{ji}.$$

\begin{corollary}\label{regular}
 Up to isomorphism, a regular  Ree's matrix semigroup  with an involution has the form:   $(S=\mathcal{M}^0(I, I, G, P), \ast)$, $(a)^{\ast}_{ij}=(za^{\ast})_{ji}$, $p_{ji}^{\ast}= zp_{ij}$, where $\ast$ is an involution on $G$ as well as $G^0$, $z\in Z(G)$ and $z^{\ast}=\frac{1}{z}$.
 \end{corollary}

\begin{lemma}\label{astrepr1}
Keep the above notations. Assume that $\C[S]$ is semi-simple.  For the pair of representations $\pi^l, \pi^r$ associated to an irreducible representation $\sigma$. The following  are equivalent:
\begin{itemize}
\item[(1)]  $\sigma$ is  isomorphic to a  $\ast$ representation of $G$,
\item[(2)] $\pi$ is isomorphic to  a  $\ast$ representation of $S$, for $\pi=\pi^r\simeq \pi^l$,
\item[(3)] $\check{\pi} \simeq D(\pi)\circ \ast$, $\pi=\pi^r \simeq \pi^l$.
\end{itemize}
\end{lemma}
\begin{proof}
$(1) \Rightarrow (2)$ Assume that $\sigma$ is an irreducible matrix $\ast$-representation of dimension $k$. Then $\sigma(g^{\ast})=\sigma(g)^{\ast}$, for any $g\in G$. In particular, for above $z\in Z(G)$, $\sigma(z)\sigma(z)^{\ast}=\sigma(z)\sigma(z^{\ast})=\sigma(z)\sigma(\frac{1}{z})=1$. As $\sigma$ is irreducible, $\sigma(z)$ is a unit in $\C^{\times}$.
 For any $s=(a)_{ij} \in S$, $(a)_{ij}^{\ast}=z(a^{\ast})_{ji}$,
 $$P(a)_{ij}= \bordermatrix{
  &  &        &  & j       &   &       &  \cr
 &0 & \cdots & 0& p_{1i}a &0 & \cdots & 0 \cr
 & 0 & \cdots & 0& p_{2i}a &0 & \cdots & 0 \cr
 &\vdots &    & \vdots  & \vdots &\vdots&  & \vdots \cr
 & 0 & \cdots & 0& p_{ni}a &0 & \cdots & 0 \cr
}$$
 $$ P(a^{\ast})_{ji}= \bordermatrix{
 &  &  &        & i   &      &   &       &  \cr
 &0 & \cdots & 0& p_{1j}a^{\ast} &0 & \cdots & 0 \cr
 &0 & \cdots & 0& p_{2j}a^{\ast} &0 & \cdots & 0 \cr
 &\vdots &    & \vdots  & \vdots &\vdots&  & \vdots \cr
 &0 & \cdots & 0& p_{nj}a^{\ast} &0 & \cdots & 0 \cr}.$$

$$(\pi(s))^{\ast}=[\sigma(P(a)_{ij})]^{\ast}=\begin{pmatrix}
0 & \cdots & 0& \sigma(p_{1i}a) &0 & \cdots & 0 \\
0 & \cdots & 0& \sigma(p_{2i}a) &0 & \cdots & 0 \\
\vdots &    & \vdots  & \vdots &\vdots&  & \vdots \\
0 & \cdots & 0& \sigma(p_{ni}a) &0 & \cdots & 0
\end{pmatrix}^{\ast} $$
$$=\begin{pmatrix}
0 & 0 &   \cdots  & 0  \\
\vdots & \vdots &  & \vdots \\
0 & 0 &   \cdots  & 0  \\
\sigma(p_{1i}a)^{\ast}& \sigma(p_{2i}a)^{\ast} & \cdots & \sigma(p_{ni}a)^{\ast}\\
0 & 0 &   \cdots  & 0  \\
\vdots & \vdots &  & \vdots \\
0 & 0 &   \cdots  & 0
\end{pmatrix}$$
$$= \begin{pmatrix}
  0 & 0 &   \cdots  & 0  \\
  \vdots & \vdots &  & \vdots \\
  0 & 0 &   \cdots  & 0  \\
 \sigma(a^{\ast}p^{\ast}_{1i})& \sigma(a^{\ast}p^{\ast}_{2i}) & \cdots & \sigma(a^{\ast}p^{\ast}_{ni})\\
  0 & 0 &   \cdots  & 0  \\
 \vdots & \vdots &  & \vdots \\
  0 & 0 &   \cdots  & 0
  \end{pmatrix}$$
  $$= \bordermatrix{
  &  &        &  &      &   &       &  \cr
  & 0 & 0 &   \cdots  & 0   \cr
 & \vdots & \vdots &  & \vdots  \cr
 & 0 & 0 &   \cdots  & 0   \cr
j& \sigma(za^{\ast}p_{i1})& \sigma(za^{\ast}p_{i2}) & \cdots & \sigma(za^{\ast}p_{in}) \cr
 & 0 & 0 &   \cdots  & 0   \cr
 &\vdots & \vdots &  & \vdots  \cr
 & 0 & 0 &   \cdots  & 0
  \cr}$$
  $$= \sigma(z(a^{\ast})_{ji} P)=\sigma(z)\sigma((a^{\ast})_{ji}) \sigma(P).$$

 $$\pi(s^{\ast})=\sigma\big(zP(a^{\ast})_{ji}\big)$$
 $$= \bordermatrix{
  &       &        &         &       i                &       &        &    \cr
  &0      & \cdots & 0       & \sigma(zp_{1j}a^{\ast}) &0      & \cdots & 0  \cr
  &0      & \cdots & 0       &\sigma( zp_{2j}a^{\ast}) &0      & \cdots & 0   \cr
  &\vdots &        & \vdots  & \vdots                 &\vdots &        & \vdots \cr
  &0      & \cdots & 0       & \sigma(zp_{nj}a^{\ast}) &0      & \cdots & 0      \cr}$$
$$=\sigma(z)\sigma(P)\sigma((a^{\ast})_{ji}).$$ Hence $(\pi(s))^{\ast}=\sigma(P)^{-1}\pi(s^{\ast})\sigma(P)$, for any $s\in S$. Note that $\sigma(z)^{\ast}=\sigma(z)^{-1}$, a unit in  $ \C^{\times}$. Let $\omega\in \C^{\times}$, such that $\omega^2= \sigma(z)$,  $\omega\omega^{\ast}=1$. Let us define a matrix $Q=\sigma(P)\omega$.  Then:
$$[\sigma(p_{ij})\omega]^{\ast}=\omega^{\ast} \sigma(p_{ij})^{\ast}=\omega^{\ast} \sigma(p_{ij}^{\ast})=\omega^{\ast} \sigma(p_{ji})\sigma(z)= \sigma(p_{ji})\omega.$$
 Therefore, $Q$ is a non-singular Hermitian  matrix, and $Q=AA^{\ast}$, for some $A\in \GL_{kn}(\C)$. Let us define $\pi'(s)=A^{-1}\pi(s)A$, for any $s\in S$. Then $\pi'(s^{\ast})=A^{-1}\pi(s^{\ast})A=A^{-1} \sigma(P) \pi(s)^{\ast} \sigma(P)^{-1}A=A^{-1} Q \pi(s)^{\ast} Q^{-1}A=A^{\ast}\pi(s)^{\ast}(A^{\ast})^{-1}=\pi'(s)^{\ast}$. Hence $\pi'$ is a matrix $\ast$-representation.  \\
$(2) \Rightarrow(1)$ Assume that $\pi$ is isomorphic to a (matrix) $\ast$-representation $\pi'$. Then $\tr \pi(s^{\ast})=\tr \pi'(s^{\ast})=\tr \pi'(s)^{\ast}=\tr \pi(s)^{\ast}$, for any $s\in S$. Assume $p_{ji}\in G$. Then for any $g\in G$, $s=(p_{ji}^{-1}g)_{ij}$, $$\tr \pi(s)^{\ast}=\tr \sigma(p_{ji} p_{ji}^{-1}g)^{\ast}=\tr \sigma(g)^{\ast}$$
$$\tr \pi(s^{\ast})=\tr\sigma(zp_{ij}[p_{ji}^{-1}g]^{\ast})=\tr\sigma(zp_{ij}g^{\ast}[p_{ji}^{-1}]^{\ast})$$
$$=\tr\sigma(p_{ij}g^{\ast}p_{ij}^{-1})=\tr\sigma(g^{\ast}).$$
Hence  $\tr \sigma(g)^{\ast}=\tr\sigma(g^{\ast})$, for any $g\in G$. Hence $D(\sigma)\circ \ast \simeq \check{\sigma}$. By Lemma \ref{finast}, $\sigma$ is isomorphic to a $\ast$-representation.\\
$(1) \Rightarrow(3)$  Assume that $\sigma$ is an irreducible matrix $\ast$-representation of dimension $k$. By the isomorphism,  we assume $\check{\sigma}(g)=\sigma(g)^{\ast T}=\sigma(g^{\ast})^T$,  for any $g\in G$. Hence for any $s=(a)_{ij}$, $\check{\pi}(s)=\check{\sigma}(P(a)_{ij})$, and $D(\pi)(s^{\ast})=\pi(s^{\ast})^T=\sigma(P(a)_{ij}^{\ast})^T=\sigma(P(za^{\ast})_{ji})^T$.
$$\check{\pi}(s)=\check{\sigma}(P(a)_{ij})$$
$$=  \bordermatrix{
  &  &        &  & j       &   &       &  \cr
  &0 & \cdots & 0& \check{\sigma}(p_{1i}a) &0 & \cdots & 0 \cr
  &0 & \cdots & 0& \check{\sigma}(p_{2i}a) &0 & \cdots & 0 \cr
  &\vdots &    & \vdots  & \vdots &\vdots&  & \vdots \cr
  &0 & \cdots & 0& \check{\sigma}(p_{ni}a) &0 & \cdots & 0\cr} $$
$$=\bordermatrix{
  &  &        &  & j       &   &       &  \cr
  &0 & \cdots & 0& \sigma(a^{\ast}p^{\ast}_{1i})^T &0 & \cdots & 0 \cr
  &0 & \cdots & 0&\sigma(a^{\ast}p^{\ast}_{2i})^T &0 & \cdots & 0 \cr
  &\vdots &    & \vdots  & \vdots &\vdots&  & \vdots \cr
  &0 & \cdots & 0& \sigma(a^{\ast}p^{\ast}_{ni})^T &0 & \cdots & 0\cr}$$
  $$=\begin{pmatrix}
  0 & 0 &   \cdots  & 0  \\
  \vdots & \vdots &  & \vdots \\
  0 & 0 &   \cdots  & 0  \\
  \sigma(a^{\ast}p^{\ast}_{1i})& \sigma(a^{\ast}p^{\ast}_{2i}) & \cdots & \sigma(a^{\ast}p^{\ast}_{ni})\\
  0 & 0 &   \cdots  & 0  \\
  \vdots & \vdots &  & \vdots \\
  0 & 0 &   \cdots  & 0\end{pmatrix}^T$$
$$=\begin{pmatrix}
0 & 0 &   \cdots  & 0  \\
\vdots & \vdots &  & \vdots \\
0 & 0 &   \cdots  & 0  \\
\sigma(za^{\ast}p_{i1})& \sigma(za^{\ast}p_{i2}) & \cdots & \sigma(za^{\ast}p_{in})\\
0 & 0 &   \cdots  & 0  \\
\vdots & \vdots &  & \vdots \\
0 & 0 &   \cdots  & 0
\end{pmatrix}^T$$
$$= \sigma((za^{\ast})_{ji} P)^T= \sigma(P)^T\sigma((za^{\ast})_{ji})^T.$$
$$D(\pi)(s^{\ast})=\sigma(P(za^{\ast})_{ji})^T=\bordermatrix{
  &  &        &  & i       &   &       &  \cr
 &0 & \cdots & 0& \sigma(zp_{1j}a^{\ast}) &0 & \cdots & 0 \cr
 &0 & \cdots & 0&\sigma(z p_{2j}a^{\ast}) &0 & \cdots & 0 \cr
 &\vdots &    & \vdots  & \vdots &\vdots&  & \vdots \cr
&0 & \cdots & 0& \sigma(zp_{nj}a^{\ast}) &0 & \cdots & 0}^T$$
$$=[\sigma(za^{\ast})_{ji}]^T\sigma(P)^T.$$
So $D(\pi)(s^{\ast})=(\sigma(P)^T)^{-1}\check{\pi}(s)\sigma(P)^T$, for all $s\in S$, and $D(\pi)\circ \ast \simeq \check{\pi}$. \\
$(3) \Rightarrow(1)$ $\tr \check{\pi}(s)=\tr D(\pi)(s^{\ast})$, for any $s=(a)_{ij} \in S$. By above discussion,
$$\tr \check{\pi}(s)= \tr \check{\sigma}(p_{ji}a), \qquad\quad \tr D(\pi)(s^{\ast})= \tr \sigma(zp_{ij}a^{\ast}).$$
Assume $p_{ji}\in G$, and $a=p_{ji}^{-1}g$, for $g\in G$. Then:
 $$\tr \check{\sigma}(g)=\tr \check{\sigma}(p_{ji}a)=\tr \sigma(zp_{ij}g^{\ast}p_{ij}^{-1}z^{\ast})
 =\tr \sigma(p_{ij}g^{\ast}p_{ij}^{-1})=\tr \sigma(g^{\ast}).$$ Hence $\check{\sigma} \simeq D(\sigma)\circ \ast$. By Lemma \ref{finast}, $\sigma$ is isomorphic to a $\ast$-representation.
\end{proof}

\begin{lemma}
Keep the notations of Corollary \ref{regular}.
Every finite dimensional  representation of $(S, \ast)$ is isomorphic to  a   semiunitary  $\ast$-representation iff (1) $P$ is an inverse matrix over $\C[G]$, (2) $p_{ii}=g\in G$, for all $i\in I$,   some  $g\in G$ with $g^2z=1$ and $g^{-1}g^{\ast}\in Z(G)$, (3) $a^{\ast}=ga^{-1} g^{-1}$, for any  $a\in G$.
\end{lemma}
\begin{proof}
$(\Rightarrow)$ By Lemma \ref{comre}, $\C[S]$ is semi-simple, so $P$ is a non-singular  matrix over $\C[G]$. Let  $\pi$ be a faithful matrix semiunitary $\ast$-representation of $S$. Then $\pi(s^{\ast})=\pi(s)^{\ast}$, $\pi(ss^{\ast}s)=\pi(s)\pi(s^{\ast})\pi(s)=\pi(s)\pi(s)^{\ast}\pi(s)=\pi(s)$, which implies that $ss^{\ast}s=s$.
Let $s=(a)_{ij}\in S$. Then $ss^{\ast}s=(a)_{ij}\circ (za^{\ast})_{ji}\circ (a)_{ij}=(zap_{jj}a^{\ast}p_{ii}a)_{ij}=(a)_{ij}$. Hence $zap_{jj}a^{\ast}p_{ii}a=a$, for all $a\in G$, $i,j\in I$. Hence $p_{ii} \neq 0$. Take $a=1\in G$, $p_{ii}=g$, for all $i\in I$, and $g^2=\frac{1}{z}$. So $a^{\ast}=ga^{-1}g^{-1}$. Since $\ast$ is an involution on $G$, $g^{-1}g^{\ast}\in Z(G)$. \\
$(\Leftarrow)$ In this case, $ss^{\ast}s=s$, for any $s\in S$. By Lemma \ref{finast}, every finite dimensional representation of $G$ is isomorphic to a $\ast$-representation. Within isomorphism, we only need to consider the standard representation $\pi$ associated to an irreducible $\ast$-representation $\sigma$ of $G$. By Lemma \ref{astrepr1}, $\pi$ is isomorphic to a $\ast$-representation $\pi'$. Hence $\pi'(s)\pi'(s)^{\ast}\pi'(s)=\pi'(s)\pi'(s^{\ast})\pi'(s)=\pi'(ss^{\ast}s)=\pi'(s)$. So $\pi'$ is a semiunitary $\ast$-representation.
\end{proof}

\subsection{A principal series}\label{aprinser}
Let us recall some known results from \cite[Chap.2.6]{CP1}, \cite[Part I]{Stein}. Let $S$ be a finite semigroup, $S^1=S\cup\{1\}$(if necessary). For $a\in S$, we call $S^1a$ the left ideal generated by $a$. Similarly, call $aS^1$ the right ideal generated by $a$, and call $S^1aS^1$ the two-sided ideal generated by $a$.  Let $L_a$(resp. $R_a$, $J_a$) denote the set of generators of $S^1a$ (resp. $aS^1$, $S^1aS^1$) in $S$.   Notice that if $S^1=S\sqcup\{1\}$, then $J_1=\{1\}$, $L_1=\{1\}$, $R_1=\{1\}$ in $S^1$.

Let $\emptyset=I_0 \subsetneq I_1 \subsetneq \cdots  \subsetneq I_n=S$ be  a principal series in  $S$ such that  each $I_i$ is a two-sided ideal of $S$, and $I_{i+1}\setminus I_i=J_{a}$, for some $a\in S$. Let $G_a=L_a \cap R_a$, and $L_a=\sqcup_{i=1}^{s_a} x_i\circ_a G_a$, $R_a=\sqcup_{j=1}^{t_a} G_a\circ_a y_j$, $J_a=\sqcup_{i=1}^{s_a}\sqcup_{j=1}^{t_a} x_i\circ_a G_a \circ_a y_j$.  Let $J_a^0=J_a\cup\{0\}$.  Follow the notation of \cite[p.71]{Stein}, for $x, y\in J_a^0$,   let  us define $y\Diamond x =\left\{\begin{array}{lr} yx & \textrm{ if } yx\in J_a\\
0& \textrm{ if } yx\notin J_a\end{array}\right.$. Then $J_a^0$ is a finite semigroup.  In particular, let  $p_{ji}= y_j \Diamond x_i$. If $y_j \Diamond x_i \in J_a $, then $ p_{ji} \in J_a \cap  aS^1\cap S^1a= G_a$; if $p_{ji} \notin J_a$, then $p_{ji}\notin G_a$. Hence $p_{ji}\in G_a^0$.  These elements $p_{ji}$  form a matrix  $P=(p_{ji})$ of type $t_a\times s_a$,   called a sandwich matrix of $J_a$.  Let $I=\{1, \cdots, s_a\}$, $J=\{1, \cdots, t_a\}$. Then Ree's theorem tells us that
$J_a^0 \simeq \mathcal{M}^0(I, J, G_a, P)$. More precisely, this isomorphism can be given as follows:

For $b=x_i\circ_a g_b\circ_a y_j\in J_a$, let it map to the  element $(g_b)_{ij}$ of $\mathcal{M}^0(I, J,G_a,  P)$, and let $0$ map to the zero matrix. If  $c=x_k\circ_a g_c\circ_a y_l$,  then $b\Diamond c=x_i\circ_a g_b\circ_a y_jx_k\circ_a g_c\circ_a y_l=x_i\circ_a g_b\circ_a p_{jk} \circ_a g_c\circ_a y_l$. \footnote{ Let us show that $b\Diamond c \in J_a$, then $ y_jx_k\in G_a$.  For then, $S^1aS^1= S^1(b\Diamond c)S^1=S^1y_jx_kS^1$. Hence $y_jx_k\in J_a\cap aS^1\cap S^1a=G_a$. } Note that $(g_b)_{ij} \circ (g_c)_{kl}=(g_b \circ_ap_{jk}\circ_ag_c)_{il}$, which is the image of $b\Diamond c$ in $\mathcal{M}^0( I, J,G_a, P)$.

Let us change $x_i$, $y_j$ by $x_i'=x_i\circ_a g_i$, $y_j'=h_j\circ_a y_j$, for some $g_i, h_j\in G_a$.  Then $p_{ji}'=y_j'x_i'=h_j\circ_ap_{ji}\circ_a g_i$. Let $P'=(p_{ji}')$. Then $P'=\diag(h_1, \cdots, h_{t_a}) P\diag(g_1, \cdots, g_{s_a})$. Therefore $\mathcal{M}^0( I, J,G_a, P) \simeq \mathcal{M}^0( I, J,G_a, P')$.

Let us recall the definition of  the left Sch\"utzenberger representation associated to $J_a$ by following \cite{LaPe} and \cite[Chap.5]{Stein}.
For $s\in S$,  if  $ sx_i\in L_a$, then $sx_i=x_j\circ_a g_{ji}$ for a unique $g_{ji}\in G_a$; otherwise, set $g_{ji}=0$. Then we can  define a left Sch\"utzenberger representation of $S$ over $\mathbb{C}[G_a]$ associated to $J_a$ by $M_J: \mathbb{C}[S] \longrightarrow M_{s_a}(\mathbb{C}[G_a]); s \longrightarrow (g_{ji})$. Using the notations from \cite[Chap.5.5]{Stein}, for a matrix representation $\sigma$ of $G_a$, the composite of $\sigma$ and $M_J$ gives a representation of $S$, denoted by $\pi=\Ind_{G_a}(\sigma)$.  For more details for such representation, one can see \cite[Thm. 1.7, Rem.]{LaPe},  or \cite[Thm. 5.5]{Stein}.  By \cite[Thm. 1.7, Rem.]{LaPe}, if $J_a$ is a regular $\mathcal{J}$-class, $\Ind_{G_a}(\sigma)$ is a proper representation; otherwise, there exists a counterexample given in the remark there. Let us see the behavior of the restriction  of  $\pi$ to $J_a^0$. For $b=x_i\circ_a g_b\circ_a y_j\in J_a$, $bx_k=x_i\circ_a g_b\circ_a y_jx_k=x_i\circ_a g_b\circ_a p_{jk}$.  Therefore $g_{ik}= g_b\circ_a p_{jk}$,   $M_J(b)= (g_{ik})=(g_b)_{ij} P$, $\sigma(M_J(b))=\sigma((g_b)_{ij} P)$. Therefore, the restriction of  $\pi$ to $J_a^0 \simeq \mathcal{M}^0( I, J,G_a, P)$ is just  a left standard  representation of $\mathcal{M}^0(I, J, G_a, P)$. In particular, for the $\sigma$ associated to $G_a$ in Def.\ref{lr}, the restriction of  $\pi$ to $J_a^0 \simeq \mathcal{M}^0( I, J,G_a, P)$ is just  a left matrix Sch\"utzenberger representation(cf. Def.\ref{lr}) of $\mathcal{M}^0(I, J, G_a, P)$.

Let us also consider the right  Sch\"utzenberger representation associated to $J_a$ by following \cite{LaPe} and \cite{Stein}. For   any $s\in S$,    if $y_j s\in R_a$, then $y_js=h_{ji} \circ_a y_i$, for a unique $h_{ji} \in G_a$; otherwise, set $h_{ji}=0$.  A right Sch\"utzenberger representation of $S$  is given  by $M_J: \mathbb{C}[S] \longrightarrow M_{s_a}(\mathbb{C}[G_a]); s \longrightarrow (h_{ji})$. The composite of $\sigma$ and $M_J$ gives a representation of $S$, denoted by $\pi'=\Coind_{G_a}(\sigma)$. Let us see the behavior of  the restriction of  $\pi'$ to $J_a^0$. For above $b\in J_a$, $y_kb=y_kx_i\circ_a g_b\circ_a y_j=p_{ki}\circ_a g_b\circ_a y_j$. Hence $h_{kj}=p_{ki}\circ_a g_b$, $M_J(b)= (h_{kj})=P(g_b)_{ij} $, $\sigma(M_J(b))=\sigma(P(g_b)_{ij} )$.  The restriction of  $\pi'$ to $J_a^0 \simeq \mathcal{M}^0( I, J,G_a, P)$ is   a right standard  representation. For the $\sigma$ associated to $G_a$ in Def.\ref{lr}, the restriction of  $\pi$ to $J_a^0 \simeq \mathcal{M}^0( I, J,G_a, P)$ is  a right matrix Sch\"utzenberger representation(cf. Def.\ref{lr}) of $\mathcal{M}^0(I, J, G_a, P)$.

  Let   $\mathcal{A}=\{ (\sigma_1, V_1), \cdots,  (\sigma_k, V_k)\}$  be a set of all pairwise inequivalent irreducible  matrix  representations of $G_a$.
 \begin{definition}
 If  $\sigma=\oplus_{i=1}^k \sigma_i$,  we will call $\pi=\Ind_{G_a}(\sigma)$ (resp. $\pi'=\Coind_{G_a}(\sigma)$)  a left (resp. right)matrix Sch\"utzenberger representation of $S$, associated to $J_a$.
 \end{definition}
 \begin{remark}
For above $\sigma$,  the restriction of $\Ind_{G_a}(\sigma)$(resp. $\Coind_{G_a}(\sigma)$) to $S_a=\mathcal{M}^0(I, J, G_a, P)$ is just the corresponding left (resp. right)matrix Sch\"utzenberger representation of $S_a$.
 \end{remark}
 The following  is one main result of Munn's paper \cite{Mu78}. For our purpose,  we need the  matrix version without  isomorphism. We shall follow the proof of \cite[Lmm.7]{Mu78}.
\begin{lemma}\label{inv}
Keep the above notations. Assume now that $S$ is an inverse semigroup. If
 \begin{itemize}
 \item[(1)] $a$ is an idempotent element $e$,
 \item[(2)]   $P$ is the identity  matrix,
  \item[(3)]  $\sigma$ is a unitary matrix representation,
  \end{itemize}
  then  $\pi=\Ind_{G_e}(\sigma)$, $\pi'=\Coind_{G_e}(\sigma)$ both are  semiunitary  representations.
\end{lemma}
\begin{proof}
Let us write $L_e=\sqcup_{i=1}^{s} x_i\circ_e G_e$, $R_e=\sqcup_{j=1}^{s} G_e\circ_e y_j$, $J_e=\sqcup_{i=1}^{s}\sqcup_{j=1}^{s} x_i\circ_e G_e \circ_e y_j$; $y_jx_i=\delta_{ij} e$. Write also $e_i=x_i\circ_e y_i=x_iy_i$. Then $e_i^2=x_iy_ix_iy_i=x_iey_i=e_i\in E(S)$, and $e_1, \cdots, e_s$ are all idempotent elements in $J_e$.  For each $e_i$, $e_i x_k=x_iy_i x_k=\delta_{ik}x_i$. For $\pi$,  $M_{J_e}: e_i \longrightarrow (e)_{ii}$.  Let $E_i=\pi(e_i)= \bordermatrix{
  &   & i  &  \cr
  &   &  &  \cr
 i&   & \sigma(e) &\cr
  &   &  &  \cr}$. Hence $\sum_{i=1}^s \pi(e_i)=\sum_{i=1}^s E_i=E$, and $\pi(s)=E\pi(s)E=\sum_{i,j}^s \sigma(e_ise_j)$, for  $s\in S$. For any $s\in S$, $\pi(s^{\ast})=\sum_{i,j}^s \sigma(e_is^{\ast}e_j)=\sum_{i,j}^s \sigma(e_i^{\ast}s^{\ast}e_j^{\ast})=\sum_{i,j}^s\sigma([e_jse_i]^{\ast})=\sum_{i,j}^s\sigma([e_jse_i])^{\ast}=\pi(s)^{\ast}$. As $ss^{\ast}s=s$, $\pi(s)\pi(s)^{\ast}\pi(s)=\pi(s)\pi(s^{\ast})\pi(s)=\pi(ss^{\ast}s)=\pi(s)$. Clearly, $\pi$ is a semiunitary representation. The proof for $\pi'$ is similar.
\end{proof}

\section{The results}\label{theresults}
Let $S$ be a finite  regular semigroup.
\begin{theorem}\label{invconditions1}
$S$ is an inverse semigroup iff   there exists a pair of  semiunitary  matrix  Sch\"utzenberger representations of $S$ associated to  each $\mathcal{J}$-class.
\end{theorem}
\begin{theorem}\label{invconditions2}
 Let $\ast$ be an involution on $S$. Then  $S$ is an inverse semigroup with the inverse map given by $\ast$ iff there exists a pair of  semiunitary   matrix  Sch\"utzenberger $\ast$-representations of $S$ associated to  each $\mathcal{J}$-class.

\end{theorem}
\begin{remark}\label{invcoro1}
There exists at most one involution on $S$ such that $S$ is an inverse semigroup with  the inverse map given by such involution.
\end{remark}

\begin{theorem}\label{fini}
 Let $\ast$ be an involution on $S$. Then every finite dimensional  representation of  $(S, \ast)$ is  isomorphic to a $\ast$-representation iff $\C[S]$ is a semi-simple  algebra and $\check{\pi} \simeq D(\pi)\circ \ast$, for any $(\pi, V)\in  \Irr(S)$.
\end{theorem}
For the  compact semigroup, one can see  \cite{HaHaSt1}, \cite{HaHaSt2}.
\section{The proofs}\label{proofs}
\subsection{The proof of  the theorem \ref{invconditions1}: } After Munn's result(Lemma \ref{inv}), it suffices to show the $``\Leftarrow''$ part. Let $\emptyset=I_0 \subsetneq I_1 \subsetneq \cdots  \subsetneq I_n=S$ be  a principal series in  $S$. Assume  $I_{i}\setminus I_{i-1}=J_{a}$. Keep the notations of Section \ref{aprinser}.  Let $S_a=\mathcal{M}^0( I, J, G_a, P)$. Assume that $\pi=\Ind_{G_a}(\sigma)$, $\pi'=\Coind_{G_a}(\sigma)$ both are  semiunitary  Sch\"utzenberger representations. Then the restrictions of $\pi$, $\pi'$  to $S_a$ are also  semiunitary representations. By Lemma \ref{astrepr}(1), $S_a$ is an inverse semigroup.   So for any $b\in S_a$, there exists a unique element, say $b^{\ast}$, such that $bb^{\ast}b=b$, and $b^{\ast}bb^{\ast}=b^{\ast}$. If there exists another $b'\in S$ such that $bb'b=b$, and $b'bb'=b'$, then $S^1bS^1=S^1bb'bS^1\subseteq S^1b'S^1=S^1b'bb'S^1\subseteq S^1bS^1$. Hence $b'\mathcal{J}b$, $b'\in J_a$. By the uniqueness in $J_a$, $b'=b^{\ast}$.
\subsection{The proof of  the theorem \ref{invconditions2}: }
By above theorem, $S$ is an inverse semigroup. By Lemma \ref{astrepr}, for any $b\in J_a$, $b^{\ast}$ is the inverse of $b$ in $J_a$ as well as $S$. Hence the result holds.
\subsection{The proof of Remark \ref{invcoro1}}
Assume that $\ast$, $\ast'$ are two involutions on $S$ such that $(S, \ast)$, $(S, \ast')$ both are inverse semigroups with the inverse maps $\ast$ and $\ast'$ respectively.  For $e\in E(S)$, $e^{\ast}=e=e^{\ast'}$. For $g\in G_e$, $g^{\ast}=g^{-1}=g^{\ast'}$. Assume $L_e=\sum_{i=1}^{s_e} x_i\circ_e G_e$. Then  $L_e^{\ast}=R_{e^{\ast}}=R_e=L_e^{\ast '}$, $L_e^{\ast}= \sum_{i=1}^{s_e} G_e\circ_e x_i^{\ast}=\sum_{i=1}^{s_e} G_e\circ_e x_i^{\ast'}=L_e^{\ast '}$. For $j\neq i$,  $gx_j^{\ast}x_i\notin J_e$, and $gx_i^{\ast} x_i=g$, for $g\in G_e$. So $x_i^{\ast}$ is the unique element of $R_e$, such that $x_i^{\ast} x_i=e$. Hence $x_i^{\ast}=x_i^{\ast'}$, $(gx_i)^{\ast}=x_i^{\ast}g^{-1}=x_i^{\ast'}g^{-1}=(gx_i)^{\ast'}$. Since $S$ is  regular, every element lies in   $L_{f}$, for some $f\in E(S)$. Hence the result holds.

\subsection{The proof of the theorem \ref{fini}.  }
For $e\in E(S)$, $S^1e=S^1x$ iff $e^{\ast}S^1=x^{\ast}S^1$, so $(L_e)^{\ast}=R_{e^{\ast}}$. Similarly, $(R_e)^{\ast}=L_{e^{\ast}}$. Then $\ast: G_e \simeq G_{e^{\ast}}$. For $s\in S$, $e\notin S^1sS^1$ iff $e^{\ast}\notin S^1s^{\ast}S^1$. Hence $s\in I_e=\{ a\in S\mid e\notin S^1aS^1\}$ iff $s^{\ast} \in I_{e^{\ast}}=\{ a\in S\mid e^{\ast}\notin S^1aS^1\}$.

$``\Rightarrow ":$ By Lemma \ref{comre}, $\C[S]$ is a  semi-simple algebra. Let $(\pi, V)$ be an irreducible representation of $S$ having an apex $e$.  Assume  that $\pi=\Ind_{G_e} \sigma, V=\Ind_{G_e} W$. Then $\check{\pi}=\Ind_{G_e} \check{\sigma}$,  $\check{V}=\Ind_{G_e} \check{W}$; $D(\pi)=D(\sigma)\otimes_{\C[G_e]} \C[R_e]$, $D(V)=D(W)\otimes_{\C[G_e]} \C[R_e]$.

 Let us write $L_e=\sqcup_{i=1}^{s_e} x_i\circ_e G_e$, $R_e=\sqcup_{j=1}^{t_e} G_e\circ_e y_j$, $J_e=\sqcup_{i=1}^{s_e}\sqcup_{j=1}^{t_e} x_i\circ_e G_e \circ_e y_j$,  $p_{ji}= y_j \Diamond x_i \in G_e^0$. These elements $p_{ji}$  form  a sandwich matrix  $P=(p_{ji})$ of type $t_e\times s_e$. Since $\C[S]$ is semi-simple, $P$ is a non-singular matrix over $\C[G_e]$. Recall that  $J_e^0$  is isomorphic to $S_e= \mathcal{M}^0( I, J, G_e, P)$ by the previous construction.

Assume that $(\pi, V)$ is isomorphic to a $\ast$-representation $(\pi', V')$.  Then $I_e=\Ann_S(V)=\Ann_S(V')=(I_e)^{\ast}=I_{e^{\ast}}$, $eV'\neq 0$ implies that $e^{\ast}V'\neq 0$.  Hence $J_e=J_{e^{\ast}}$.
Hence $\ast$ defines an involution on $J_e^0$. The restriction  of $\pi$ to $J_e^0$ or $S_e$ is also isomorphic to a  $\ast$-representation.  By Lemma \ref{astrepr1}, $\check{\pi}|_{J_e^0} \simeq [D(\pi)\circ \ast]|_{J_e^0}$.  Note that $\check{\pi}$ and $D(\pi)\circ \ast$ both are irreducible representations having the apex $e$. Applying the character theory of semigroups to our situation (cf. \cite[Prop. 7.14]{Stein} or \cite{MaQuSt}), $\check{\pi}\simeq D(\pi)\circ \ast$ by \cite[Thm.5.5]{Stein}.

$``\Leftarrow ":$ By semi-simplicity, it reduces to  the irreducible matrix representations.   For then, keep the above notations. Assume $V=\C^m$.  Then   $(\check{\pi}, \check{V})$ or  $(D(\pi), D(V))$ has  an apex $e$.  Hence for the left representation $(D(\pi)\circ \ast, D(V))$, $\Ann_{S}(  D(V))=(I_{e})^{\ast}=  I_{e^{\ast}}$, and $[D(\pi)\circ \ast] (e^{\ast}) (D(V))\neq 0$. So $D(\pi)\circ \ast$ has an apex $e^{\ast}$. Hence  $I_e=I_{e^{\ast}}$, and $e\mathcal{J} e^{\ast}$. Moreover, $J_{e}=J_{e^{\ast}}$. By Lemma \ref{astrepr1}, the restriction  of $\pi$ to $J_e^0$  is isomorphic to a  $\ast$-representation. Hence $\tr \pi(s)=\tr \pi(s^{\ast})^{\ast}$, for any $s\in J_e^0$. Let $\pi''(s)=\pi(s^{\ast})^{\ast}$, which also defines an irreducible representation of $S$. For $s\in S$, $\pi''(s)=0$ iff $\pi(s^{\ast})=0$ iff $s\in (I_e)^{\ast}=I_e$. Moreover, $\pi''(e)=\pi(e^{\ast})^{\ast}\neq 0$. Hence $\pi''$ also  has an apex $e$.  Applying the character theory(cf. \cite[Prop. 7.14]{Stein} or \cite{MaQuSt}, \cite[Thm.5.5]{Stein}) to the equality: $\tr \pi(s)=\tr \pi(s^{\ast})^{\ast}$, $s\in J_e^0$, we obtain $\pi''\simeq \pi$.  Therefore there exists a non-singular matrix $A\in \GL_m(\C)$, such that $A\pi(s^{\ast})=\pi(s)^{\ast}A$, for all $s\in S$.  Then  $\pi(s^{\ast})=A^{-1}\pi(s)^{\ast}A$, $\pi(s^{\ast})^{\ast}=A^{\ast}\pi(s)(A^{-1})^{\ast}$, $\pi(s)^{\ast}=A^{\ast}\pi(s^{\ast})(A^{-1})^{\ast}$. Therefore,
$\pi(s^{\ast})=A^{-1}\pi(s)^{\ast}A=A^{-1}A^{\ast}\pi(s^{\ast})(A^{-1})^{\ast}A$. Since $\pi$ is irreducible, $A^{-1}A^{\ast}=c\Id$, for some $c\in \C^{\times}$, and $\|c\|=1$. Assume $c=d/\overline{d}$. Then $(dA)^{\ast}=\overline{d}cA=dA$. By replacing $A$ with $dA$,  we assume that  the above $A$ is a Hermitian matrix. Let  $A=B^{\ast}B$, for some $B\in \GL_m(\C)$. Put $\pi'(s) =B\pi(s)B^{-1}$. Then $\pi'(s^{\ast})=B\pi(s^{\ast})B^{-1}=BA^{-1}\pi(s)^{\ast}AB^{-1}=(B^{-1})^{\ast}\pi(s)^{\ast}B^{\ast}=\pi'(s)^{\ast}$. So $\pi'$ is a $\ast$-representation.

\setcounter{secnumdepth}{5}

 \labelwidth=4em
\addtolength\leftskip{25pt}
\setlength\labelsep{0pt}
\addtolength\parskip{\smallskipamount}

\end{document}